\documentclass{article}
\usepackage{amsmath,amsfonts,amsthm}
\usepackage{graphicx}
\usepackage{color}
\usepackage{enumerate}

\DeclareMathOperator{\walk}{walk}

\newcommand{\alg}[1]{\mathbf{#1}}
\newcommand{\relstr}[1]{\mathbb{#1}}
\newcommand{\tuple}[1]{\mathbf{#1}}
\renewcommand{\alph}{\mathcal A}
\def\centerby#1#2{%
  \setbox0=\hbox{#1}%
  \hbox to \wd0{\hss#2\hss}%
}

\newcommand{\redgev}{\mathrel{\rightarrow}}

\newcommand{\defeq}{\buildrel \rm def\over =}

\newcommand{\Z}{\mathbb Z}

\newcommand{\relpow}[1]{^{\circ #1}}

\theoremstyle{plain}
\newtheorem{theorem}{Theorem}[section]
\newtheorem{proposition}[theorem]{Proposition}

\newtheorem{corollary}[theorem]{Corollary}
\newtheorem{question}[theorem]{Question}
\newtheorem{definition}[theorem]{Definition}
\newtheorem{claim}[theorem]{Claim}
\newtheorem{lemma}[theorem]{Lemma}

\begin{document}

\title{Local loop lemma}
\author{Miroslav Ol\v s\'ak}
\maketitle

\begin{abstract}
  We prove that an idempotent operation generates a loop from a
  strongly connected digraph containing directed cycles of all lengths under very mild
  (local) algebraic assumptions. Using the result, we reprove the existence of a
  weakest non-trivial idempotent equations, and that a strongly connected
  digraph with algebraic length 1 compatible with a Taylor term has a loop.
\end{abstract}

\section{Introduction}

Theorems that give a loop in a graph under certain algebraic and
structural assumptions play an important role in universal algebra and
constraint satisfaction problem. One example of such a ``loop lemma''
is the following one.

\begin{theorem} [loop lemma]
  \label{dir-loop-lemma}
  \cite{BartoKozikLoop, BartoKozikCyclic}
  If a finite digraph $\relstr G$
  \begin{itemize}
  \item is weakly connected,
  \item is smooth (has no sources and no sinks),
  \item has algebraic length 1 (cannot be homomorphically mapped to a non-trivial
    directed cycle) and
  \item is compatible with a Taylor term,
  \end{itemize}
  then $\relstr G$ contains a loop.
\end{theorem}

The consequences of this loop lemma include the following.
\begin{itemize}
\item \cite{BartoKozikLoop} If a digraph $\relstr G$ has no sources and no sinks, and
$\relstr G$ has a component that cannot be homomorphically mapped to a
circle, then constraint satisfaction problem over $\relstr G$ is NP-complete. This was a positive answer to an influential the Hell-Ne\v set\v ril conjecture~\cite{HellNesetril} in the domain of computational complexity.
\item \cite{OptimalStrong} Every locally-finite Taylor algebra has a term operation $s$ satisfying $s(r,a,r,e) = s(a,r,e,a)$. Taylor varieties are essential in universal algebra, especially in tame congruence theory and Maltsev conditions. The fact that locally finite Taylor algebras can be characterized by such a simple condition was utterly unexpected in universal algebra, and a similar condition was later found for infinite Taylor algebras~\cite{WeakestIdempotent}.
\item \cite{BartoKozikCyclic} Every finite Taylor algebra
  $\alg A$ has cyclic terms of all prime number arities bigger than
  $|\alg A|$. This, not so direct, application of loop lemma ranks
  among the strongest characterizations of finite Taylor
  algebras.
\end{itemize}

The modern proof~\cite{BartoKozikCyclic} of the loop lemma above requires idempotency and is based on absorption. An operation $f$ is said to be idempotent if
$f(x,x,\ldots,x)=x$ for any $x$. The definition of absorption is
slightly more complex.
Let $A$ be a set, $X,Y$ subsets of $A$, and $f$ be an $n$-ary
operation on $A$. We say that $X$ \emph{absorbs} $Y$ with respect to
$f$ if for any coordinate  $i=0,\ldots,n-1$ and any elements
$x_0,x_1,\ldots,x_{i-1},y,x_{i+1},\ldots,x_{n-1}\in A$
such that $y\in Y$ and each $x_j\in X$, we have
$$
t(x_0,x_1,\ldots,x_{i-1},y,x_{i+1},\ldots,x_{n-1})\in X.
$$

Another loop lemma based on absorption, which was used for the
proof that there are the weakest non-trivial idempotent
equations~\cite{WeakestIdempotent} and which drops the finiteness
assumption, has the following form.
\begin{theorem}
\label{absorbtion-loop-lemma}
Let $\relstr G$ be an undirected, not necessarily finite graph that
contains an odd cycle and is compatible with an idempotent
operation $f$. Assume that for every non-isolated node
$x\in\relstr G$, the set of neighbors of $x$ absorbs $\{x\}$ with
respect to $f$. Then $\relstr G$ has a loop.
\end{theorem}

The absorption assumption in Theorem~\ref{absorbtion-loop-lemma} is
not particularly strong, it is weaker than compatibility with NU term,
or absorption of a diagonal by the edges of $\relstr G$, see
Proposition 4.5 in~\cite{WeakestIdempotent}.
On the other hand, it still requires some form of homogeneity --
it have to be satisfied for every non-isolated node $x$, and the
definition of absorption hides another universal quantifiers
inside. The idea that such level of homogeneity may not be necessary
was expressed by the following question in~\cite{WeakestIdempotent}
\begin{question}
\label{local-loop-question}
Let $\relstr G$ be an undirected graph,
containing a cycle
of odd length with an element $a$. Moreover let $f$ be an
idempotent operation compatible with $\relstr G$ such that the
neighborhood of the node $a$ absorbs $\{a\}$ with respect to $f$.
Does $\relstr G$ necessarily contain a loop?
\end{question}

A slight progress in this area was made before. L. Barto has found
a proof for finite set $A$, and also a general proof in the case of
cycle of length 3 was found.
The main result of this paper is a version of loop lemma under even
significantly weaker assumptions than the original question. That
makes our ``local loop lemma'' one of the strongest, even among finite
loop lemmata.
\begin{theorem}[local loop lemma]
\label{main-thm}
Consider a set $A$, operation $t\colon A^n\to A$ a digraph
$\relstr G$ on $A$, and vertices $\alpha_{i,j}\in\relstr G$ for
$i,j\in\{0,\ldots,n-1\}$ such that
\begin{enumerate}
\item $t$ is idempotent,
\item $\relstr G$ is compatible with $t$,
\item $\relstr G$ is either a strongly connected digraph containing
directed cycles of all lengths starting with two, or $\relstr G$ is an undirected
connected non-bipartite graph.
\item for every $i \in \{0,\ldots,n-1\}$, there is a $\relstr G$-edge
  \[
  \alpha_{i,i} \redgev t(\alpha_{i,0},
  \alpha_{i,1}, \ldots, \alpha_{i,n-1})
  \]
\end{enumerate}
Then $\relstr G$ contains a loop.
\end{theorem}

\begin{proof}[Proof of positive answer to Question~\ref{local-loop-question}]
Consider an element
$x\in A$ in an odd cycle such that the neighborhood of
$x$ absorbs $\{x\}$. Then the component of $x$ is closed under $t$
(see Corollary~\ref{component-closed} for detailed explanation), so
we can restrict to that component.
The item (4) is satisfied by putting $\alpha_{i,i} = x$ and $\alpha_{i,j}=y$
otherwise, where $y$ is any element in the neighborhood of $x$. Then
$t(y,\ldots,y,x,y,\ldots,y)$ is in the neighborhood by absorption, so
$$
\alpha_{i,i} = x \redgev t(y,\ldots,y,x,y,\ldots,y)
= t(\alpha_{i,0},\ldots,\alpha_{i,n-1}).
$$
\end{proof}

Note that the absorption approach is not the only one widely used to
tackle loop lemmata. The oldest method is based on performing
pp-definitions and pp-interpretations mostly on the graph side. This
resulted in older, weaker versions of Theorem~\ref{dir-loop-lemma},
see~\cite{BulatovLoop,HellNesetril}, but also provides a state-of-the-art
version of loop lemma for oligomorphic
structures~\cite{PseudoSiggers}. The most recent technique is based on
the correspondence of certain loop lemmata with Maltsev
conditions,
see~\cite{LoopConditions,LoopConditionsDirected,PseudoLoopConditions}.
However, none of the methods available are local enough for our
purposes. Therefore, we have chosen a different approach, a blindly
straightforward one. The only thing we actually do is that we define
what exactly to plug into a star-power of the operation $t$ to get a
loop. Yet, such an approach appears to be among the most efficient
ones.

\subsection{Outline}

In section~\ref{preliminaries} we give proper definitions of the used
terms, alongside with our notation for finite sequences that will be
used in the main proof.
In section~\ref{main-thm-sec} we prove our main result,
Theorem~\ref{main-thm}. In section~\ref{double-loop}, we prove a stronger local version of the existence of a weakest non-trivial equations, the main result from~\cite{WeakestIdempotent}. In
section~\ref{taylor-sec}. we get further strengthening of
Theorem~\ref{main-thm} that yields a version of
Theorem~\ref{dir-loop-lemma} with slightly stronger relational
assumption (strongly connected digraph) but slightly weaker algebraic
assumptions.
In section~\ref{conclusion}, we discuss possible further
generalizations of the main result.

\section{Preliminaries and notation}
\label{preliminaries}

\subsection{Words, integer intervals}

Consider a set $\alph$ representing an alphabet. By a word, we mean a
finite sequence of elements in $\alph$. The set of all
words in the alphabet $\alph$ of length $n$ is denoted by $\alph^n$.
For manipulation with words, we use a Python-like syntax.
\begin{itemize}
\item $\mathbf x = [a_0, a_1, \ldots, a_{n-1}]$ represents a word of
  length $n$, The length of $\mathbf x$ is denoted by $|\mathbf x|$.
\item Elements of the word $\mathbf x$ can be extracted using an index
  in round brackets after the word, that is $\mathbf x[i] = a_i$. The
  first position is indexed by zero. By a \emph{position} in a word
  we mean an integer that represents a valid index.
\item For $0\leq i\leq j\leq n$ we define a \emph{subword}
  \[\mathbf x[i:j] = [a_i, a_{i+1}, \ldots, a_{j-1}].\]
  Notice that the interval includes $i$ and does not include $j$.
\item If $i$ or $j$ is omitted, the boundaries of the word are
  used. That is
  \[
  \mathbf x[i:] = [a_i, \ldots, a_{n-1}], \quad \mathbf x[:j] = [a_0,\ldots,a_{j-1}]
  \]
\item Inspired by the subword notation, we use single $[i:j]$ to
  represent an integer interval. That is
  $[i:j] = \{i,i+1,\ldots,j-1\}$, where $i,j$ can be arbitrary
  integers. Notice that $i$ is included in that interval while $j$ is
  not. If $i$ is omitted, it is meant implicitly as zero, that is
  $[:n] = \{0,1,\ldots,n-1\}$. The set of all integers is denoted by
  $\Z$.
\item Words can be concatenated using symbol $+$, that is
  \[[a_0,\ldots,a_{n-1}]+[b_0,\ldots,b_{m-1}] =
  [a_0,\ldots,a_{n-1},b_0,\ldots,b_{n-1}].\].
\end{itemize}

A word $\mathbf x$ is said to be periodic with a period $k\geq 1$, or
briefly $k$-periodic, if
$\mathbf x[i] = \mathbf x[i+k]$ whenever both $i$ and $i+k$ are valid
indices. Alternatively speaking, $\mathbf x\in\alph^n$ is $k$-periodic
if $k\geq n$ or $\mathbf x[:n-k] = \mathbf x[k:]$. A 1-periodic word
is also called a constant word.
In our proof, we use the following well-known property of
periodic words.
\begin{proposition}[Periodicity lemma]
  \label{periodic-lemma}
  Let $a,b$ be positive integers and $\mathbf x$ be a word of length
  at least $a+b-\gcd(a,b)$. If $\mathbf x$ is both $a$-periodic and
  $b$-periodic, it is also $\gcd(a,b)$-periodic.
\end{proposition}
\begin{corollary}
  \label{subword-period}
  Let $\mathbf x$ be a word and $k\geq 2$ be the shortest period of
  $\mathbf x$. If $\mathbf y$ is a subword of $\mathbf x$ such that
  $|\mathbf y|\geq 2k-2$. Then $k$ is the shortest period of
  $\mathbf y$.
\end{corollary}
\begin{proof}
  The word $\mathbf y$ is $k$-periodic. To obtain a contradiction, let
  $k' \leq k-1$ be another period of $\mathbf y$. Since
  $|\mathbf y|\geq k'+k-1$, the word $\mathbf y$ is
  $\gcd(k,k')$-periodic. Since $|\mathbf y|\geq k$ and $\mathbf y$ is
  a subword of the $k$-periodic word $\mathbf x$, also $\mathbf x$ is
  $\gcd(k,k')$-periodic, which contradicts the minimality of $k$.
\end{proof}

\subsection{Operations, star powers}

An $n$-ary operation $t$ on a set $A$ is a mapping
$t\colon A^n\to A$. Instead
of $t([a_0,\ldots,a_{n-1}])$, we simply write $t(a_0,\ldots,a_{n-1})$.
An operation $t$ on $A$ is said to be \emph{idempotent} if
$t(x,x,\ldots,x) = x$ for every $x\in A$.

For a given $n$-ary operation $t$ and a non-negative integer $k$ we
recursively define $k$-th star power of $t$, denoted $t^{*k}$,
to be $n^k$-ary operation given by
\begin{align*}
t^{*0}(x) &= x, \\
t^{*(k+1)}(x_0,\ldots,x_{n^{k+1}-1}) &=
t(t^{*k}(x_0,\ldots,x_{n^k-1}),\ldots,t^{*k}(x_{(n-1)n^k},\ldots,x_{n^{k+1}-1})).
\end{align*}
We perceive variables in star powers as being indexed by words in
$[:n]^k$, where the left-most letters corresponds to the outer-most
position in the composition tree. More precisely, a substitution of
variables in a star power $t^{*k}$ is represented by a function
$f\colon [:n]^{k}\to A$. If $k=0$, then $t^{*0}(f) = f([])$.
Otherwise, we can compute $t^{*k}(f)$ by
\[
t^{*k}(f) = t\bigl(t^{*k}(f_0), t^{*k}(f_1), \ldots, t^{*k}(f_{n-1})\bigr)
\text{ or }
t^{*k}(f) = t^{*(k-1)}(f'),
\]
where $f_i, f'\colon [:n]^{k-1}$ are defined by
\[
f_i(\mathbf x) = f([i]+\mathbf x)\text{ and }
f'(\mathbf x) = t\bigl(f(\mathbf x+[0]), f(\mathbf x+[1]), \ldots, f(\mathbf x+[n-1])\bigr).
\]

\subsection{Algebras, equations}

A signature $\Sigma$ is a set of symbols accompanied with their
arities. An abstract algebra $\alg A = (A, t_0, t_1, \ldots)$ in the
signature $\Sigma$ is a set $A$ together with representations of symbols in
$\Sigma$ as actual \emph{basic} operations on $A$ of the
corresponding arities. A term in a signature $\Sigma$ is a
syntactically valid expression using the term symbols of $\Sigma$ and
variables. A term operation in $\alg A$ is an operation on $A$ that
can be written as a term in $\Sigma$, represented by basic operations
in $\alg A$.

An equation in $\Sigma$ is a pair of terms in $\Sigma$ written as
$t_0 \approx t_1$. An \emph{equational condition} $\mathcal C$ is a system of
equations in any signature, say $\Delta$. An algebra $\alg A$ is said to
satisfy an equational condition $\mathcal C$ if it is possible to
assign some term operations in $\alg A$ to the symbols in $\Delta$ so
that all the equations in $\mathcal C$ hold for any choice of
variables in $\alg A$.

Equational conditions are thoroughly studied in universal algebra in
the form of strong Maltsev conditions (equational conditions
consisting of finitely many equations) and Maltsev conditions
(infinite disjunction of strong Maltsev conditions).
Of particular interest are the Taylor equations that represents the
weakest non-trivial idempotent Maltsev condition. The signature
consists of a
single $n$-ary symbol $t$. The \emph{Taylor system of equations} is any
system of $n$ equations of the form
\begin{align*}
t(x,?,?,\ldots,?,?) &\approx t(y,?,?,\ldots,?,?), \\
t(?,x,?,\ldots,?,?) &\approx t(?,y,?,\ldots,?,?), \\
&\vdots\\
t(?,?,?,\ldots,?,x) &\approx t(y,?,?,\ldots,?,y), \\
t(x,x,x,\ldots,x,x) &\approx x,
\end{align*}
where each question mark stands for either $x$ or $y$. A \emph{Taylor
  operation} is any operation satisfying any Taylor system of equations.
A \emph{quasi Taylor} system of equations is a
Taylor system of equations without the last one requiring
idempotency. For the purposes of our proofs, we enumerate the first $n$ (quasi)
Taylor equations from top to bottom by integers
from $0$ to $n-1$.
For more background on universal algebra, we refer
the reader to~\cite{Bergman}.
Note that it was recently
proved~\cite{WeakestIdempotent} that the weakest non-trivial
idempotent Maltsev condition can be written in a specific form of a
strong Maltsev condition. We reprove this fact in
Section~\ref{double-loop}.

\subsection{Relations, digraphs}

An $n$-ary relation on a set $A$ is any subset of $A^n$. A relation
$R$ is said to be compatible with an $m$-ary operation $t$, if for any
tuple of words $\mathbf r_0, \mathbf r_1, \ldots, \mathbf r_{m-1}\in R$, the
the result of $t(\mathbf r_0, \mathbf r_1, \ldots, \mathbf r_{m-1})$ is
in $R$ as well, where the operation $t$ is applied to
$\mathbf r_0,\ldots \mathbf r_{m-1}$ point-wise. A relation is said to
be compatible with an algebra $\alg A$ if it is compatible with all
basic operations of $\alg A$, or algebraically said, if it is a subuniverse of an
algebraic power $\alg A^n$. Notice that if a relation $R$ is
compatible with an algebra $\alg A$, it is compatible with all term
operations in $\alg A$. In particular, if $R$ is compatible with an
operation $t$, then $R$ is compatible with any star power of $t$.

A relational structure $\relstr R = (A,R_0,R_1\ldots)$ on $A$ is the set
$A$ together with a collection of relations $R_0, R_1 \ldots$ on $A$. An
algebra $\alg A$, or an operation $t$ on $A$ is compatible with a relational
structure $\relstr R$ on $A$, if $\alg A$, or $t$, is compatible
with all the relations in $\relstr R$.

A digraph $\relstr G = (V,E)$ is a relational structure with a single
binary relation. If the set $E$ of edges is symmetric, we call the
digraph an \emph{undirected graph}.
Given a digraph $\relstr G = (V,E)$, we usually denote the edges by
$u\redgev v$ instead of $[u,v]\in E$.
By a \emph{$n$-walk} from $v_0$ to $v_n$, or simply a
\emph{walk}, we mean a sequence
(word) of vertices in the digraph
\[
[v_0, v_1, \ldots, v_n]
\]
such that $v_i\redgev v_{i+1}$ for all $i\in[:n]$. While we use most
of the notation we have for words also for walks, we redefine a length
of a walk to be $n$, that is one less than the length of the
appropriate word of vertices.
A \emph{cycle walk} of length $n$, or $n$-cycle walk, is such an
$n$-walk $\mathbf w$ that $\mathbf w[0] = \mathbf w[n]$.

The $n$-th relational power $\relstr G\relpow n$ of a digraph
$\relstr G$ is a digraph with the same set of vertices, and
$u\redgev v$ in $\relstr G\relpow n$ if and only if there is a
$n$-walk in $\relstr G$ from $u$ to $v$. Notice that if a
digraph $\relstr G$ is compatible with an algebra $\alg A$, any
relational power of $\relstr G$ is compatible with $\alg A$ as well.

A digraph is said to be strongly connected, if there is a walk from
$u$ to $v$ for any pair of vertices $u,v$. We say that a digraph have
an algebraic length 1, if it cannot be homomorphically mapped to a
directed cycle of length greater than one.

We finish this chapter by proving basic combinatorial properties of
strongly connected graphs of algebraic length 1.
\begin{proposition}
  \label{large-enough-cycles}
  Let $u$ be a vertex in a strongly connected digraph $\relstr G$ with
  algebraic length $1$. Then there are directed cycle walks
  containing $u$ of any large enough length.
\end{proposition}
\begin{proof}
  First observe that if all cycle walks in $\relstr G$ are divisible
  by some $n\geq 2$ and $\relstr G$ is strongly connected, then
  $\relstr G$ can be homomorphically mapped to the directed cycle of
  length $n$.

  Therefore, since $\relstr G$ is supposed to have has algebraic
  length 1, there are some cycle walks
  $\mathbf c_0, \ldots \mathbf c_{k-1}$ such that the greatest common
  divisor of the lengths of the cycles equals one. Let
  $\mathbf w_i$ denote a walk from $u$ to $\mathbf c_i[0]$ and
  $\mathbf w'_i$ denote a walk from $\mathbf c_i[0]$ to $u$ for every
  $i\in[:k]$. There is a cycle walk starting in $u$ of
  any length of the form
  \[
  |\mathbf w_0| + x_0|\mathbf c_0| + |\mathbf w'_0|
  + |\mathbf w_1|  + x_1|\mathbf c_1| + |\mathbf w'_1|
  + \ldots +
  |\mathbf w_{k-1}| + x_{k-1}|\mathbf c_{k-1}| + |\mathbf w'_{k-1}|,
  \]
  where $x_0,x_1,\ldots,x_{k-1}$ stands for any non-negative integer
  coefficients. Since
  $\gcd(|\mathbf c_0|,\ldots,|\mathbf c_{k-1}|)=1$, this number can
  reach any large enough integer.
\end{proof}
\begin{proposition}
  \label{large-enough-walks}
  If $\relstr G$ is a finite strongly connected digraph with algebraic
  length 1, then there is an integer $K$ such that there is a $k$-walk
  from $v_0$ to $v_1$ for any $v_0,v_1\in\relstr G$ and $k\geq K$.
\end{proposition}
\begin{proof}
  Let $d(v_0, v_1)$ denote the length of the shortest walk from $v_0$ to
  $v_1$. Let $d$ be the largest $d(v_0,v_1)$
  among all pairs of vertices $v_0, v_1\in \relstr G$.
  Fix an element $u\in\relstr G$. By
  Proposition~\ref{large-enough-cycles}, there is such a length $C$
  that there is a $c$-cycle walk from $u$ to $u$ of any length $c\geq C$.
  Thus the
  choice $K = d+C+d$ works for any pair $v_0, v_1$ since
  \[
  k \geq K = d+C+d \geq d(v_0,u) + C + d(u, v_1).
  \]
\end{proof}
\begin{proposition}
  \label{dir-component-closed}
  Let $t$ be an idempotent $n$-ary operation compatible with a graph
  $\relstr G$. Let $\relstr H\subset\relstr G$ be a strongly
  connected component of $\relstr G$ that have an algebraic length
  1. Then $\relstr H$ is closed under $t$.
\end{proposition}
\begin{proof}
  Fix a vertex $u\in\relstr H$. By
  Proposition~\ref{large-enough-cycles}, there is such a length $C$
  that there are $c$-cycle walks from $u$ to $u$ of any length
  $c\geq C$.
  Consider any $v_0,\ldots,v_{n-1}\in\relstr H$. We prove that there
  is a walk from $u$ to $t(v_0,\ldots,v_{n-1})$.

  Let $\mathbf w_i$ denote a walk from $u$ to $v_i$. There are also
  walks from $u$ to $v_i$ of a fixed length
  \[
  k = C + \max(|\mathbf w_0|, |\mathbf w_1|, \ldots, |\mathbf w_n|).
  \]
  Therefore, there is a $k$-walk from $u$
  to $t(v_0,\ldots,v_{n-1})$ in $\relstr G$
  since $t$ is idempotent and $\relstr G\relpow k$ if compatible with
  $t$.
  The existence of the walk in the other direction is
  analogous. Hence $t(v_0,\ldots,v_{n-1})\in\relstr H$. Since
  $v_0,\ldots,v_{n-1}$ can be chosen arbitrarily, $\relstr H$ is
  closed under $t$.
\end{proof}
\begin{corollary}
  \label{component-closed}
  Let $\relstr H$ be a non-bipartite connected component of an
  undirected graph $\relstr G$ compatible with an idempotent operation
  $t$. Then $\relstr H$ is closed under the operation $t$.
\end{corollary}

\section{Proof of the local loop lemma}
\label{main-thm-sec}

We prove the local loop lemma in the following form.

\begin{theorem}[local loop lemma]
  \label{local-loop-basic}
  Consider a set $A$, operation $t\colon A^n\to A$ a digraph
  $\relstr G$ on $A$, and elements $\alpha_{i,j}$ for
  $i,j\in\{0,\ldots,n-1\}$ such that
  \begin{enumerate}
  \item $t$ is idempotent,
  \item $\relstr G$ is compatible with $t$,
  \item $\relstr G$ is a strongly connected digraph containing
    cycle walks of all lengths greater that one,
  \item for every $i \in \{0,\ldots,n-1\}$, there is a $\relstr G$-edge
    \[
    \alpha_{i,i} \redgev
    t(\alpha_{i,0}, \alpha_{i,1}, \ldots, \alpha_{i,n-1})
    \]
  \end{enumerate}
  Then $\relstr G$ contains a loop.
\end{theorem}

Theorem~\ref{local-loop-basic} differs from Theorem~\ref{main-thm} in
the item (3), where Theorem~\ref{main-thm} allows also an undirected connected
non-bipartite graph. We start by explaining how Theorem~\ref{main-thm}
follows from Theorem~\ref{local-loop-basic},

\begin{proof}[Proof of Theorem \ref{main-thm}]
  If $\relstr G$ has cycle walks of all lengths greater than one,
  we get a loop directly by Theorem~\ref{local-loop-basic}. Assume
  that it does not, thus $\relstr G$
  is an undirected connected non-bipartite graph.
  Therefore, there is a cycle of odd length in $\relstr G$, let us
  denote the smallest odd length of such a cycle by $l$. To obtain a
  contradiction, assume that there is no loop in $\relstr G$,
  hence $l\geq 3$. Observe that $\relstr G$ contains cycle walks
  of all lengths $l'\geq l-1$: there are cycle walks of any even
  length jumping around a single edge, and cycle walks of any odd
  length greater that $l-1$ that first go around a cycle of length $l$
  and then jumps around a single edge.

  Consider the graph
  $\relstr G' = \relstr G\relpow{(l-2)}$. By
  minimality of $l$, $\relstr G'$ does not have a loop.
  Since
  $l-2$ is an odd number and $\relstr G$ is undirected, the edges of
  $\relstr G\relpow{(l-2)}$ form a superset of the edge set of
  $\relstr G$,
  hence $\relstr G'$ satisfies the item (4) of
  Theorem~\ref{local-loop-basic} about $\alpha_{i,j}$.
  Compatibility of $\relstr G'$ with $t$, that is item (2), follows
  from basic properties of relational powers. Since
  $\relstr G$ contains a cycle walk of every length greater than $l-2$,
  the digraph $\relstr G'$ contains a cycle walk of any length greater
  than 1. Therefore, by Theorem~\ref{local-loop-basic}, there is a
  loop in $\relstr G'$ corresponding to a cycle walk of length $l-2$ in
  $\relstr G$ which contradicts the minimality of $l$.
\end{proof}

The proof of Theorem~\ref{local-loop-basic} relies on the following
technical proposition.

\begin{proposition}
  \label{f-exists}
  Let $n$ be a positive integer and $\alph$ denote $[:n]$.
  Consider a strongly connected digraph $\relstr G$ that
  contains cycle walks of all lengths, and let $\alpha_{i,j}$ be any
  vertices of $\relstr G$.
  Then there is a positive integer $N$ and a mapping
  $f\colon \alph^N \to G$ (substitution to the star power) such that
  for any $\mathbf x\in \alph^N$ one of the following cases happen:
  \begin{enumerate}
  \item there is $i\in \alph$ such that
    $f(\mathbf x) = \alpha_{i,i}$ and
    \[
    \forall j\in\alph\colon f(\mathbf x[1:]+[j]) = \alpha_{i,j},
    \]
  \item for every $i\in\alph$, there is an $\relstr G$-edge
    \[
    f(\mathbf x) \to f(\mathbf x[1:]+[i]).
    \]
  \end{enumerate}
\end{proposition}

\begin{proof}[Proof of Theorem~\ref{local-loop-basic}]
  Take the substitution function $f\colon \alph^N\to \relstr G$ given
  by Proposition~\ref{f-exists}. Based on that, we define two functions
  $f_0, f_1\colon \alph^{N+1}\to \relstr G$.
  We set $f_0(\mathbf x) = f(\mathbf x[:N])$ and
  $f_1(\mathbf x) = f(\mathbf x[1:])$. Let
  \begin{align*}
    f'_0(\mathbf x) &= t(f_0(\mathbf x+[0]), \ldots, f_0(\mathbf x+[n-1])),\\
    f'_1(\mathbf x) &= t(f_1(\mathbf x+[0]), \ldots, f_1(\mathbf x+[n-1])).\\
  \end{align*}
  By idempotency of $t$, the functions $f'_0$ and $f$ are
  identical. Therefore
  \[
  t^{*(N+1)}(f_0) = t^{*N}(f'_0) = t^{*N}(f) =
  t(t^{*N}(f),t^{*N}(f),\ldots,t^{*N}(f)) = t^{*(N+1)}(f_1)
  \]
  We claim that there is an edge from $t^{*(N+1)}(f_0)$ to
  $t^{*(N+1)}(f_1)$. We verify the edge by checking an edge
  from $f'_0(\mathbf x)$ to $f'_1(\mathbf x)$ for any $\mathbf x\in\alph^N$.
  We analyze the two cases of the behavior of $f$ on $\mathbf x$.
  \begin{enumerate}
  \item If there is $i\in \alph$ such that
    $f(\mathbf x) = \alpha_{i,i}$ and
    $f(\mathbf x[1:]+[j]) = \alpha_{i,j}$ for all $j\in\alph$, then
    $f'_0(\mathbf x) = f(\mathbf x) = \alpha_{i,i}$ and
    \[
    f'_1(\mathbf x)
    = t(f(\mathbf x[1:]+[0]), \ldots, f(\mathbf x[1:]+[n]))
    = t(\alpha_{i,0},\ldots,\alpha_{i,n-1}),
    \]
    so there is an edge $f'_0(\mathbf x) \redgev f'_1(\mathbf x)$
    by definition of $\alpha_{i,j}$.
  \item If there is an edge
    \[
    f_0(\mathbf x+[i]) = f(\mathbf x) \redgev f(\mathbf x[1:]+[i]) = f_1(\mathbf x+[i])
    \]
    for every $i\in\alph$, then by compatibility of $t$ and
    $\relstr G$, there is an edge
    $f'_0(\mathbf x) \redgev f'_1(\mathbf x)$.
  \end{enumerate}
  So there is an edge from $f'_0(\mathbf x)$ to $f'_1(\mathbf x)$
  for any $\mathbf x\in\alph^N$, consequently there is an edge from
  $t^{*(N+1)}(f_0)$ to $t^{*(N+1)}(f_1)$ and since
  $t^{*(N+1)}(f_0) = t^{*(N+1)}(f_1)$, it forms the desired loop.
\end{proof}

Proposition~\ref{f-exists} will be proved in the following two
subsections. In Subsection~\ref{construction} we explicitly define
the function $f$, in Subsection~\ref{proofs} we prove that the
function $f$ satisfy the required properties. Before that we reduce
the problem to a finite case.

\begin{lemma}
  \label{finite-subgraph}
  Let $\relstr G$ be a strongly digraph that contains a cycle
  walk of every length greater than one. Let $A$ be a finite set of
  vertices in $\relstr G$. Then there is a finite subdigraph
  $\relstr G'\subset\relstr G$ that
  is strongly connected and contains a cycle walk of every
  length greater than 1 and all elements of $A$, 
\end{lemma}
\begin{proof}
  We start with a finite subdigraph $\relstr G_0\subset\relstr G$ with
  algebraic length one -- it suffices to put any two coprime cycle
  walks of $\relstr G$ into $\relstr G_0$ and connect them. Thus
  there is a length $C$ such that $\relstr G_0$ contains a $c$-cycle
  walk for any $c\geq C$. We construct $\relstr G'$ by adding the
  following edges and nodes into $\relstr G_0$:
  \begin{itemize}
  \item one cycle walk of every length in the interval $[2:C]$,
  \item all the vertices in $A$,
  \item paths connecting the elements in previous items to a fixed
    node in $\relstr G_0$ and vice versa.
  \end{itemize}
  These are finitely many edges and vertices in total. The final
  $\relstr G'$ is therefore finite while it meets the required criteria.
\end{proof}

\subsection{Construction of the substitution $f$}
\label{construction}

In this section, we construct a witness to the
Proposition~\ref{f-exists}. In particular, we consider the digraph
$\relstr G$, positive integer $n$ and vertices $\alpha_{i,j}$ and
define an appropriate integer $N$ and a function
$f\colon \alph^N\to\relstr G$, where $\alph$ denotes $[:n]$.
By Lemma~\ref{finite-subgraph}, we
can assume that $\relstr G$ is finite.
Consequently, we can use Proposition~\ref{large-enough-walks} and get
$K$ such that
$K\geq 2$ and there are $k$-walks from $v_0$ to $v_1$ for any
$v_0,v_1\in\relstr G$ and $k\geq K$. For every $v_0,v_1,k$, we fix
such a walk and denote it by $\walk(v_0,v_1,k)$.
We define the length $N$ as $N=L+W+R$ (left, window, right),
where
\begin{itemize}
\item $W = 3K-3$,
\item $M = 2(K-1)\cdot n^W + (K-1)$.
\item $R = M+K-1$,
\item $L = R+K-2$,
\end{itemize}

The overall idea is to evaluate $f(\mathbf x)$ primarily by the
``window'' $\mathbf x[L:L+W]$, or to investigate the neighborhood of this
window, if necessary.
We start with constructing a \emph{priority function}
$\pi\colon \alph^W\to \Z$ and a \emph{value function}
$\nu\colon \alph^W\to \relstr G$ of the following properties.

\begin{enumerate}
\item\label{itm:w-constant} If $\mathbf w$ is constant $[i,\ldots,i]$, then
  $\nu(\mathbf w) = \alpha_{i,i}$ and $\pi(\mathbf w) = 0$.
\item\label{itm:w-periodic} If $\mathbf w$ is periodic with the shortest period
  $k \in [2:K]$, then $\pi(\mathbf w)=R$ and there is a $k$-cycle walk
  $[v_0,v_1,\ldots,v_k]$ in $\relstr G$ such that
  \[\nu(\mathbf w[i:]+\mathbf w[:i]) = v_i\]
  for every $i\in[:n]$.
\item\label{itm:w-almost-const} If $\mathbf w[:W-1]$ is constant but $\mathbf w$ is not, then
  $\pi(\mathbf w) = R$.
\item\label{itm:pi-negative} If $\mathbf w$ is not periodic with a period smaller than $K$
  and $\mathbf w[:W-1]$ is not constant, then $\pi(\mathbf w)$ is negative.
\item\label{itm:pi-injective} $\pi$ is ``injective on negative values'', that is, whenever
  $\pi(\mathbf w_0) = \pi(\mathbf w_1) < 0$, then
  $\mathbf w_0 = \mathbf w_1$.
\end{enumerate}
The construction of such functions is straightforward. To satisfy the
conditions \ref{itm:w-constant}, \ref{itm:w-almost-const} we simply set
the appropriate values of $\pi$ and $\nu$. The items
\ref{itm:pi-negative} \ref{itm:pi-injective} can be satisfied since
there are infinitely many negative numbers and just finitely many
possible words of length $W$. Finally, to meet the condition
\ref{itm:w-periodic}, for all $k\in[2:K]$, we partition the words with
the smallest period $k$ into groups that differs by a cyclic
shift. Any such group can be arranged as
$\mathbf w_0, \mathbf w_1, \ldots, \mathbf w_{k-1}$ where
$\mathbf w_i = \mathbf w_0[i:] + \mathbf w_0[:i]$. For every such a
group, we find a $k$-cycle $[v_0,v_1\ldots,v_k]$ in $\relstr G$ and
set $\nu(\mathbf w_i) = v_i$, $\pi(\mathbf w_i) = R$.

Now consider a word $\mathbf x\in\alph^N$ of length $N$. Positions
$p\in[:N-W+1]$ represent valid indices of subwords $\mathbf x[p:p+W]$
of length $W$. We define values $\nu_{\mathbf x}\colon
[:N-W+1]\to\relstr G$ and priorities $\pi_{\mathbf x}\colon
[:N-W+1]\to\relstr \Z$ of such positions by
\[
\nu_{\mathbf x}(p) = \nu(\mathbf x[p:p+W]),\quad
\pi_{\mathbf x}(p) = \pi(\mathbf x[p:p+W])
\]
with the following exceptions if $\mathbf x[p:p+W]$ is constant.
\begin{itemize}
\item Let $q\in [p:N-W+1]$ be the right-most position such that
  $\mathbf x[p:q+W]$ is constant. We redefine
  $\pi_{\mathbf x}(p)$ to be $\min(q-p, R-1)$ instead of zero.
\item If $p \in [1:L+1]$, $\mathbf x[p-1:p-1+W+R]$ is a constant word
  $[i,\ldots,i]$ and
  $\mathbf x[p-1+W+R] = j$, we redefine $\nu_{\mathbf x}(p)$ to be $\alpha_{i,j}$
  instead of $\alpha_{i,i}$.
\end{itemize}

Based on the priority function $\pi_{\mathbf x}\colon[:N-W+1]\to \Z$,
we define local maxima. A position $p\in [:N-W+1]$ is called a
\emph{local maximum} in $\mathbf x\in\alph^N$
if either $\pi_{\mathbf x}(p)=R$,
or $p\in [K-1:N-W-(K-1)+1]$ and
$\pi_{\mathbf x}(p)\geq \pi_{\mathbf x}(q)$
whenever $|p-q| < K$.

We are finally ready to construct the function
$f\colon \alph^N \to\relstr G$. If $L$ is a local maximum in
$\mathbf x$, we simply set $f(\mathbf x) = \nu_{\mathbf x}(L)$.
Otherwise, we find the closest local maxima to $L$ from both sides.
In particular let $p<L$ be the right-most local maximum before $L$,
and let $q>L$ be the left-most local maximum after $L$. We claim that
these positions exist and that
$q-p\geq K$, these claims are proved in the following subsection.
In that case we set
\[
f(\mathbf x) = \walk(\nu_{\mathbf x}(p), \nu_{\mathbf x}(q), q-p)[L-p].
\]

\subsection{Proofs}
\label{proofs}

In this section, we fill the missing proofs in the construction. In
particular, we prove the following.
\begin{itemize}
\item It is possible to find a local maximum on both sides of the
  position $L$ in any $\mathbf x\in \alph^N$
  (Corollary~\ref{local-max-around-L}).
\item If $L$ is not a local maximum and $p,q$ are local maxima such
that $p<L<q$, then $q-p\geq K$ (Corollary~\ref{walk-long-enough}).
\item The constructed mapping $f$ meets the criteria given by
Proposition~\ref{f-exists} (Proposition~\ref{seq-map-valid}).
\end{itemize}

\begin{lemma}
  \label{close-local-max}
  Let $p < q$ be local maxima in $\mathbf x$ such that $q-p < K$. Then
  $\nu_{\mathbf x}(p) = \nu_{\mathbf x}(q)\geq R-1$ and the segment
  $\mathbf x[p:q+W]$ is periodic with a period strictly less than $K$.
\end{lemma}
\begin{proof}
  Since both $p$, $q$ are local maxima,
  $\nu_{\mathbf x}(p) = \nu_{\mathbf x}(q)$.
  First we prove that $\nu_{\mathbf x}(p)\geq 0$.
  To obtain a contradiction, suppose that $\nu_{\mathbf x}(p) < 0$.
  Then $\mathbf x[p:p+W] = \mathbf x[q:q+W]$ by injectivity of the
  function $\nu$ on negative values. Hence $\mathbf x[p:q+W]$ is
  periodic with a period $q-p < K$. That contradicts the assumption
  that $\nu_{\mathbf x}(p) < 0$.

  Therefore $\nu_{\mathbf x}(p) \geq 0$, and both subwords
  $\mathbf w_0 = \mathbf x[p:p+W-1]$ and
  $\mathbf w_1 = \mathbf x[q:q+W-1]$ are periodic with
  periods less than $K$, let us denote their shortest periods $k_0$,
  $k_1$ respectively.
  Their intersection $\mathbf w = \mathbf x[q:p+W-1]$ has length at least
  \[
  (W-1)-(K-1) = (K-1)+(K-1)-1 \geq k_0+k_1-\gcd(k_0, k_1),
  \]
  so it is $\gcd(k_0, k_1)$-periodic by
  Proposition~\ref{periodic-lemma}.
  Since $|\mathbf w|\geq \max(k_0,k_1)$ and the subwords $\mathbf w_0$, $\mathbf w_1$ are $k_0$-periodic
  or $k_1$-periodic, they are uniquely
  determined by $\mathbf w$. Therefore, the whole subword $\mathbf x[p:q+W-1]$
  is $\gcd(k_0, k_1)$-periodic and $\gcd(k_0, k_1) = k_0 = k_1$.

  If $\mathbf x[p:q+W-1]$ is not constant, then $k_0 \geq 2$,
  and $\mathbf w_1$ is not constant. Since
  $\nu_{\mathbf x}(q)\geq 0$, the word $\mathbf x[q:q+W]$ is then periodic with
  a period less that $K$. By the same reasoning as above, the shortest
  period of $\mathbf x[q:q+W]$ is $k_0$ and the whole part
  $\mathbf x[p:q+W]$ is $k_0$-periodic.

  Otherwise, if $\mathbf x[p:q+W-1]$ is constant, then
  $R > \nu_{\mathbf x}(p) = \nu_{\mathbf x}(q)$, so $\mathbf x[q:q+W]$ is
  constant, and the whole part $\mathbf x[p:q+W]$ is constant.
  Thus $\nu_{\mathbf x}(p) = \min(R-1, \nu_{\mathbf x}(q) + q-p)$, so
  $\nu_{\mathbf x}(p) = \nu_{\mathbf x}(q) = R-1$.
\end{proof}

\begin{corollary}
  \label{walk-long-enough}
  Let $\mathbf x\in \alph^N$ be a word such that
  $L$ is not a local maximum in $\mathbf x$, and let $p,q$ be local maxima
  such that $p<L<q$. Then $q-p\geq K$.
\end{corollary}
\begin{proof}
  Conversely suppose that $q-p < K$. By~Lemma~\ref{close-local-max},
  $\mathbf x[p:q+W]$ is periodic with a period strictly less than $K$ and
  $\pi_{\mathbf x}(p) = \pi_{\mathbf x}(q) \geq R-1$.
  If $\mathbf x[p:q+W]$ is constant, $\pi_{\mathbf x}(L) = R-1$, else
  $\pi_{\mathbf x}(L) = R$. In both cases, $L$ is a local maximum
  contrary to the assumption.
\end{proof}

\begin{lemma}
  \label{local-max-check-on-right}
  Let $p\in[K-1:N-W-(K-1)+1]$ be such that $\mathbf x[p:p+W]$ is
  constant. If $p$ is not a local maximum, then one of the following
  scenarios happen.
  \begin{enumerate}
  \item $\pi_{\mathbf x}(p) < R-1$ and $\mathbf x[p-1] = \mathbf x[p]$,
  \item there is a local maximum $q > p$ such that $q-p < K$,
    $\mathbf x[p:q+W-1]$ is constant and different from
    $\mathbf x[q+W-1]$.
  \end{enumerate}
\end{lemma}
\begin{proof}
  Since $p$ is not a local maximum, there is a position $q$ such that $|p-q| < K$ and $\nu_{\mathbf x}(q) > \nu_{\mathbf x}(p)\geq 0$, hence $\mathbf x[q:q+W-1]$ is periodic with a period smaller than $K$. The subword $\mathbf x[q:q+W-1]$ has an intersection with the constant subword $\mathbf x[p:p+W]$ of length at least $(W-1)-(K-1) = 2K-3\geq K-1$. Therefore by periodicity, $\mathbf x[q:q+W-1]$ is constant as well.

  We analyze two cases by the position of $q$.
  \begin{enumerate}
  \item If $q < p$, then $q+W-1 \in [p:p+W]$, so $\mathbf x[q:q+W]$ is still constant. Therefore $\pi_{\mathbf x}(q)\leq R-1$, and consequently $\pi_{\mathbf x}(p) < R-1$. Moreover $\mathbf x[p] = \mathbf x[p-1]$ and we get the scenario (1).
  \item If $q > p$, we show that $\mathbf x[q+W-1]$
  differs from the constant on $\mathbf x[q:q+W-1]$, so the scenario (2) happens. If it did not, the whole segment $\mathbf x[p:q+W]$ would be constant, and $\nu_{\mathbf x}(p) = \min(R-1, \nu_{\mathbf x}(q) + q-p)$ would contradict $v_{\mathbf x}(q) > \nu_{\mathbf x}(p)$.
  \end{enumerate}
\end{proof}

\begin{corollary}
  \label{local-max-in-const}
  Let $p\in[K-1:N-W-(K-1)+1]$ be such a position that the segment
  $\mathbf x[p:p+W+(K-1)]$ is constant. Assume that
  $\mathbf x[p] \neq \mathbf x[p-1]$ or $\nu_{\mathbf x}(p) = R-1$.
  Then $p$ is a local maximum.
\end{corollary}

\begin{lemma}
  \label{local-max-nonconst-interval}
  Consider positions $p_0, p_1\in[K-1:N-W+1]$ in a word
  $\mathbf x$ such that $p_1-p_0 \geq M$. If $\mathbf x[p_0:p_1+W]$ is not constant,
  there is a local maximum in $\mathbf x$ in the interval $[p_0:p_1+1]$.
\end{lemma}
\begin{proof}
  If there is no such a position $q\in[p_0:p_0+2(K-1)\cdot n^W+1]$ that
  $\mathbf x[q:q+W]$ is constant, we find the local maximum by the
  following process.
  We start with the position $q_0 = p_0+(K-1)n^W$. While $q_i$ is not a local
  maximum, we find $q_{i+1}$ such that $|q_{i+1} - q_i| \leq K-1$ and
  $\pi_{\mathbf x}(q_{i+1}) > \pi_{\mathbf x}(q_i)$. Observe that the positions
  $q_1, q_2, \ldots, q_{n^W}$ cannot escape the interval
  $[p_0:p_0+2(K-1)\cdot n^W+1]$.
  On the other hand, the process cannot have more than $n^W$ steps
  since the values $\pi_{\mathbf x}(p_i)$ form an increasing sequence
  which is made of at most $n^W$ negative values and one non-negative
  value $R$. So we will get to the local maximum eventually.

  If there is a position $q\in[p:p+2(K-1)\cdot n^W+1]$ such that
  $\mathbf x[q:q+W]$ is constant, we find $q_0, q_1$ such that
  $p_0\leq q_0\leq q \leq q_1 \leq p_1$ and
  $\mathbf x[q_0:q_1+W]$ is the largest possible constant segment
  containing the position $q$.
  If $q_1 < p_1$, then $\pi_{\mathbf x}(q_1+1) = R$, hence
  $q_1+1$ is a local maximum in $[p_0:p_1+1]$.
  Assume otherwise that $q_1 = p_1$.
  Since $\mathbf x[p_0:p_1+W]$ is not constant, we have
  $q_0 > p_0$. Since $p_1-q_0\geq M - 2(K-1)\cdot n^W = K-1$, $q_0$ is the desired
  local maximum by Corollary~\ref{local-max-in-const}.
\end{proof}

\begin{corollary}
\label{local-max-big-interval}
Consider a position $p\in[K-1:L+2]$ in a word $\mathbf x$. Then
there is a local maximum in the interval $[p:p+(R-1)+1]$.
\end{corollary}
\begin{proof}
  If $\mathbf x[p:p+R-1+W]$ is not constant, there is a local maximum by
  Lemma~\ref{local-max-nonconst-interval} since $R-1 \geq M$.
  If $\mathbf x[p:p+R-1+W]$ is constant, then $\nu_{\mathbf x}(p) = R-1$
  and $p$ is a local maximum by Corollary~\ref{local-max-in-const}.
\end{proof}

\begin{corollary}
  \label{local-max-around-L}
  Let $\mathbf x\in \alph^N$ be a word such that
  $L$ is not a local maximum in $\mathbf x$. Then there are a local
  maxima $p,q$ in $\mathbf x$ such that $p<L<q$.
\end{corollary}
\begin{proof}
  We find $p$ in the interval $[L-R+1:L+1]$ and $q$ in the
  interval $[L:L+R]$ by Corollary~\ref{local-max-big-interval}.
\end{proof}

Now, we are going to prove that the constructed mapping $f$ satisfies
given conditions. For that purpose, we investigate how functions
$\pi_{\mathbf x}, \nu_{\mathbf x}$ relates to functions
$\pi_{\mathbf y}, \nu_{\mathbf y}$, where $\mathbf y = \mathbf x[1:]+[i]$
for some $i$.

\begin{lemma}
  \label{value-shift}
  Let $\mathbf x,\mathbf y\in\alph^N$ be words such that
  $\mathbf y[:N-1] = \mathbf x[1:]$, and
  $p\in[2:N-W+1]$. Then $\nu_{\mathbf x}(p) = \nu_{\mathbf y}(p-1)$
  or $\mathbf x[L:N]$ is constant and $p = L+1$.
\end{lemma}
\begin{proof}
  Clearly, $\mathbf x[p:p+W] = \mathbf y[p-1:(p-1)+W]$, so
  \[
  \nu\bigl(\mathbf x[p:p+W]\bigr) = \nu\bigl(\mathbf y[p-1:(p-1)+W]\bigr).
  \]
  To confirm the lemma,
  it remains to discuss the exceptional behavior of $\nu$ that assigns $\alpha_{i,j}$.
  Fix $i,j\in\alph$. We claim that
  with the exception of $p=L+1$ and $\mathbf x[L:N]$ being constant,
  the following items are satisfied
  \begin{itemize}
  \item $p\leq L$,
  \item $\mathbf x[p-1:p-1+W+R]$ is a constant word $[i,\ldots,i]$,
  \item $\mathbf x[p-1+W+R] = j$.
  \end{itemize}
  if and only if the following items are satisfied
  \begin{itemize}
  \item $p-1\leq L$,
  \item $\mathbf y[p-2:p-2+W+R]$ is a constant word $[i,\ldots,i]$,
  \item $\mathbf y[p-2+W+R] = j$.
  \end{itemize}
  The forward implication is clear. The only case in which the backward one could fail is
  when $p-1 \leq L$ but $p\not\leq L$, that is $p = L+1$. In that case, since $\mathbf y[p-2:p-2+W+R]$ is constant, we get that
  \[
  \mathbf y[p-2:p-2+W+R] = \mathbf y[L-1:L-1+W+R]
  = \mathbf x[L:L+W+R] = \mathbf x[L:N]
  \]
  is constant as well.
\end{proof}

\begin{lemma}
  \label{priority-shift}
  Let $\mathbf x,\mathbf y\in\alph^N$ be words such that
  $\mathbf y[:N-1] = \mathbf x[1:]$, and $p\in[1:N-W+1]$.
  If $p > L+1$ and $\mathbf y[p-1:N]$ is constant, then
  $\pi_{\mathbf y}(p-1) = \pi_{\mathbf x}(p)+1 \in[0:R]$.
  Otherwise $\pi_{\mathbf y}(p-1) = \pi_{\mathbf x}(p)$.
\end{lemma}
\begin{proof}
  If $\mathbf x[p:p+W]$ is not constant, then
  \[
  \pi_{\mathbf x}(p) = \pi(\mathbf x[p:p+W])
  = \pi(\mathbf y[p-1:p-1+W]) = \pi_{\mathbf y}(p-1).
  \]
  If $\mathbf x[p:p+W]$ is constant but $\mathbf y[p-1:N]$ is not,
  there is $q\in[p-1+W:N-1]$. such that
  $\mathbf y[p-1:q]$ is constant but $\mathbf y[q+1]\neq\mathbf y[q]$.
  Therefore $\mathbf x[p:q+1]$ is constant and
  \[
  \pi_{\mathbf x}(p) = \min((q+1)-(p+W), R-1)
  = \min(q-(p-1+W), R-1) = \pi_{\mathbf y}(p-1).
  \]
  If $\mathbf y[p-1:N]$ is constant and $p \leq L+1$, then
  $\pi_{\mathbf x}(p) = R-1 = \pi_{\mathbf y}(p-1)$.
  Finally, if $\mathbf y[p-1:N]$ is constant and $p > L+1$,
  then
  \[
  \pi_{\mathbf y}(p-1) = N-(p-1) = (N-p)+1
  = \pi_{\mathbf x}(p)+1\in[0:R].
  \]
\end{proof}

\begin{lemma}
  \label{local-max-shift}
  Let $\mathbf x,\mathbf y\in\alph^N$ be words such that
  $\mathbf y[:N-1] = \mathbf x[1:]$, and
  $p\in[K:N-W+1]$.
  If $p$ is a local maximum in
  $\mathbf x$, then $p-1$ is a local maximum in $\mathbf y$.
  Conversely, if $p-1$ is a local maximum in $\mathbf y$ and $p$ is
  not a local maximum in $\mathbf x$, then $p \geq L+2$ and there is a
  local maximum in $\mathbf x$ in the interval $[L+1:p]$.
\end{lemma}
\begin{proof}
  By Lemma~\ref{priority-shift},
  $\pi_{\mathbf x}(p) = R$ if and only if $\pi_{\mathbf y}(p-1) = R$, in
  that case both $p$ and $p-1$ are local maxima.
  Assume otherwise, that is $\pi_{\mathbf x}(p) < R$,
  and $\pi_{\mathbf y}(p) < R$.

  We first prove the forward implication by contradiction. Suppose that $p$ is a local maximum in $\mathbf x$ but $p-1$ is not a local maximum in $\mathbf y$. We thus find a position $q$ such that $|p-q|<K$,
  $\pi_{\mathbf x}(p)\geq \pi_{\mathbf x}(q)$ and $\pi_{\mathbf y}(p-1) < \pi_{\mathbf y}(q-1)$. By Lemma~\ref{priority-shift}, the priority can raise by at most one when shifting to the left. Since the inequality changed, it was an equality before, that is $\pi_{\mathbf x}(p) = \pi_{\mathbf x}(q)$, and then the priority changed at $q$ but not at $p$:
  $\pi_{\mathbf y}(p-1) = \pi_{\mathbf x}(p)$ and
  $\pi_{\mathbf y}(q-1) = \pi_{\mathbf x}(q)+1$. 
  Moreover, since the priority changed at $q$, $\mathbf y[q-1:N]$ must be constant.
  Since $\pi_{\mathbf y}(p-1) = \pi_{\mathbf x}(q)\in [0:R]$, also
  $\mathbf y[p-1:p-1+W]$ is constant, and consequently, $\mathbf
  y[p-1:N]$ is constant. Since $\pi_{\mathbf y}(q-1) > \pi_{\mathbf
    y}(p-1)$, we get $q < p$. On the other hand, since the priority
  increased at $q$ but not at $p$, we get $p\geq L+1>q$ by
  Lemma~\ref{priority-shift}. Satisfying both is impossible.
  
  Now we prove the other part. Let us assume that $p-1$ is a local
  maximum in $\mathbf y$ but $p$ is not a local maximum in
  $\mathbf x$.
  There are two possible reasons for $p$ not being a local maximum in
  $\mathbf x$. Either $p > N-W-(K-1)$, or there is a position $q$ such
  that $\nu_{\mathbf x}(q) > \nu_{\mathbf x}(p)$ and $|p-q| < K$.

  If $p > N-W-(K-1)$, then $p = N-W-(K-1)+1$ and $\mathbf y[p-2:(p-1)+W]$ is not constant since $p-1$ is a local
  maximum in $\mathbf y$. Therefore $\mathbf x[L+1:p+W]$ is not constant, so there
  is a local maximum in the interval $\mathbf x[L+1:p+1]$ by
  Lemma~\ref{local-max-nonconst-interval} since $p-(L+1) = R-K+1 = M$.
  However, $p$ is not a local maximum, so the maximum belongs to the
  interval $[L+1:p]$.

  Now suppose that there is $q$ such that
  $\nu_{\mathbf x}(q) > \nu_{\mathbf x}(p)$, and $|p-q| < K$. Since $p-1$
  is a local maximum in $\mathbf y$,
  $\nu_{\mathbf y}(q-1) \leq \nu_{\mathbf y}(p-1)$.
  Similarly as in the forward implication, we obtain the following identities from Lemma~\ref{priority-shift},
  \[
  \nu_{\mathbf x}(q) = \nu_{\mathbf y}(q-1) = \nu_{\mathbf y}(p-1) =
  \nu_{\mathbf x}(p)+1.
  \]
  Moreover $\nu_{\mathbf y}(p-1)\in [0:R]$ and $\mathbf y[p-1:N]$ is constant since $\nu_{\mathbf y}(p-1)\neq\nu_{\mathbf x}(p)$. Since $\nu_{\mathbf x}(q) = \nu_{\mathbf y}(q-1) = \nu_{\mathbf y}(p-1)$, $\mathbf x[q:q+W]$ is constant and $\nu_{\mathbf x}(q) = R-1$. We compute
  \[
  \nu_{\mathbf x}(p) = \nu_{\mathbf y}(p-1)-1 = \nu_{\mathbf x}(q) - 1 = (R-1)-1 = R-2,
  \]
  therefore $p = N-(R-2) = L+2$. On the other hand, $q\leq L+1$ since
  $\nu_{\mathbf x}(q) = \nu_{\mathbf y}(q-1)$. Therefore $[L+1:N]$ is constant and $L+1$ is a local maximum in $[L+1:p]$.
\end{proof}

\begin{proposition}
  \label{seq-map-valid}
  The function $f\colon \alph^N\to\relstr G$, as constructed in
  subsection~\ref{construction} is such that for any
  $\mathbf x\in \alph^N$, one of the following cases happen:
  \begin{enumerate}
  \item There is $i\in \alph$ such that
    $f(\mathbf x) = \alpha_{i,i}$ and $f(\mathbf x[1:]+[j]) = \alpha_{i,j}$
    for all $j\in\alph$,
  \item for every $i\in\alph$, there is an edge in $\relstr G$
    \[
    f(\mathbf x) \redgev f(\mathbf x[1:]+[i]).
    \]
  \end{enumerate}
\end{proposition}
\begin{proof}
  If $\mathbf x[L:N]$ is constant, the first case happens. In
  particular, $i = \mathbf x[L]$, $\pi_{\mathbf x}(L) = R-1$,
  $L$ is a local maximum, $\nu_{\mathbf x}(L) = \alpha_{i,i}$, so
  $f(\mathbf x) = \alpha_{i,i}$. Let $\mathbf y$ denote
  $\mathbf x[1:]+[j]$. Thus $\pi_{\mathbf y}(L) = R-1$,
  $L$ is a local maximum in $\mathbf y$ and we have $\nu_{\mathbf x}(L) = \alpha_{i,j}$ by the exceptional case for
  value $\nu_{\mathbf y}$.

  If $\mathbf x[L:N]$ is not constant, we show that the second case
  happens. Let $\mathbf y = \mathbf x[1:]+[i]$.
  First, we prove it if both $L$ and $L+1$ are local maxima
  in $\mathbf x$. In this case $L-1$ and $L$ are local maxima in
  $\mathbf y$ by Lemma~\ref{local-max-shift}. Also
  $\nu_{\mathbf y}(L) = \nu_{\mathbf x}(L+1)$ by Lemma~\ref{value-shift}.
  By Lemma~\ref{close-local-max}, $\pi_{\mathbf x}(L) = \pi_{\mathbf x}(L+1)\geq 0$
  and $\mathbf x[L:L+W+1]$ is periodic with a period smaller than $K$.
  We show that $\mathbf x[L:L+W+1]$ cannot be constant.
  Assume that the subword is constant to obtain a contradiction, then
  $\pi_{\mathbf x}(L) = \min(R-1, \pi_{\mathbf x}(L+1)+1)$. Since
  $\pi_{\mathbf x}(L) = \pi_{\mathbf x}(L+1)$, we get
  $\pi_{\mathbf x}(L+1) = R-1$, so $\mathbf x[L+1:(L+1)+W+(R-1)]$ is
  constant. That contradicts the assumption that $\mathbf x[L:N]$ is not
  constant. Therefore $\mathbf x[L:L+W+1]$ is periodic with a smallest
  period $k$ such that $1<k<K$. Thus $k$ is also the smallest period
  of words 
  $\mathbf x[L:L+W]$ and $\mathbf x[L+1:L+1+W]$
  by Corollary~\ref{subword-period}. By definition of
  $\nu$, the vertices $\nu(\mathbf x[L:L+W])$ and $\nu(\mathbf
  x[L+1:L+W+1])$ are consecutive vertices on a cycle walk of length
  $k$, so there is an edge
  \[
  f(\mathbf x) = \nu(\mathbf x[L:L+W]) \redgev
  \nu(\mathbf x[L+1:L+W+1]) = f(\mathbf y).
  \]

  Now, let us assume that $L$ or $L+1$ is not a local maximum in $\mathbf x$.
  Let $p_0$ be the right-most local maximum such that $p_0\leq L$, and let
  $p_1$ be the left-most local maximum such that $p_1 > L$.
  They both exist by Corollary~\ref{local-max-big-interval}, and
  $p_1 - p_0 \geq K$ by Corollary~\ref{walk-long-enough}. By the choice of
  $p_0, p_1$, there is no local maximum strictly between $p_0, p_1$.
  Therefore by Lemma~\ref{local-max-shift} $p_0-1, p_0-1$ are local maxima
  in $\mathbf y$ and there is no local maximum between them.
  Since $\mathbf x[L:N]$ is not constant,
  $\nu_{\mathbf y}(p_0-1) = \nu_{\mathbf x}(p_0)\defeq u_0$
  and $\nu_{\mathbf y}(p_1-1) = \nu_{\mathbf x}(p_1)\defeq u_1$
  by Lemma~\ref{value-shift}. Finally, we get the desired edge
  \begin{align*}
  f(\mathbf x) = \nu_{\mathbf x}(L) ={}& \walk(u_0, u_1, p_1-p_0)[L-p_0]\\ \redgev{}& \walk(u_0, u_1, p_1-p_0)[L-p_0+1] = \nu_{\mathbf y}(L) = f(\mathbf y).
  \end{align*}
\end{proof}

\section{Double loop}
\label{double-loop}

The core of the paper describing the weakest nontrivial
equations~\cite{WeakestIdempotent} is the proof that the existence of
a Taylor term implies the existence of a double loop term,
that is a term $d$ satisfying the double loop equations:
\begin{align*}
d(xx,xxxx,yyyy,yy) &= d(xx,yyyy,xxxx,yy)\\
d(xy,xxyy,xxyy,xy) &= d(yx,xyxy,xyxy,yx)
\end{align*}
The variables are grouped together for better readability.
The double loop equations can be obtained as follows. Consider a
$4 \times 12$ matrix whose columns are all the four-tuples
$[a_0,a_1,b_0,b_1] \in \{x,y\}^4$ with $a_0 \neq a_1$ or
$b_0 \neq b_1$, and let
$\tuple{r}_0,\tuple{r}_1,\tuple{r}_2,\tuple{r}_3$ denote
its rows. The double loop equations are then $d(\tuple{r}_0) \approx
d(\tuple{r}_1)$ and $d(\tuple{r}_2) \approx d(\tuple{r}_3)$. If the
columns are organized lexicographically with $x < y$, we get the
equations above.

The fact that a Taylor term implies a double loop term is proved
in~\cite{WeakestIdempotent} by an intermediate step in a form of a
double loop lemma. We provide a local version of
that procedure. Not only the local loop lemma makes possible to get the
double loop lemma in a more straightforward manner but also the
implication Taylor term $\Rightarrow$ double loop term gets a
stronger, ``local'' notion: if an idempotent algebra satisfies Taylor
equations locally on $X$, it satisfies the double loop equations
locally on $X$. The notion of locally satisfied equational condition
is defined below.

\begin{definition}
  Let $\alg A$ be an algebra with a subset $X\subset\alg A$.
  Let $\mathcal S$  be an equational condition.
  We say that $\alg A$ satisfies $\mathcal S$ locally on $X$
  if it is possible to assign term operations in $\alg A$ to the
  term symbols in $\mathcal S$ so that every equation is satisfied
  whenever the variables are chosen from the set $X$.
\end{definition}
Note that if $X$ is the universe of $\alg A$, then $\alg A$ satisfying
$\mathcal S$ locally on $X$ just means that $\alg A$ satisfies
$\mathcal S$ as an equational condition.

\begin{theorem}[Local double loop lemma]
  \label{double-loop-lemma}
  Let  $\alg A = (A; t^\alg A)$ and $\alg B = (B; t^\alg B)$ be
  algebras in the signature consisting of a single $n$-ary operation
  symbol $t$. Assume that $\alg A$ is generated by
  $\{x^\alg A,y^\alg A\}$, $t^\alg A$ is idempotent, $\alg B$ is
  generated by $\{x^\alg B,y^\alg B\}$ and $t^\alg B$ satisfy the
  quasi Taylor system of equations locally on
  $\{x^\alg B,y^\alg B\}$.
  Let $Q$ be the subuniverse of $\alg A^2\times \alg B^2$ generated by
  all the 12 quadruples $[a_0,a_1,b_0,b_1]$ with
  $a_0,a_1\in\{x^\alg A,y^\alg A\},\; b_0,b_1\in\{x^\alg B,y^\alg B\}$,
  such that $a_0\neq a_1$ or $b_0\neq b_1$.
  Then there is a double loop in $Q$, that is, a quadruple
  $[a,a,b,b]\in Q$.
\end{theorem}
\begin{proof}
  We assume that $x^\alg A\neq y^\alg A$ and $x^\alg B\neq y^\alg B$,
  otherwise the theorem is trivial.
  Let us define a graph $\relstr G = (A,E)$ on $A$ by
  \[
  E = \{[a_0,a_1]\in A^2 \mid \exists b\in B\colon [a_0,a_1,b,b]\in Q.\}
  \]
  Observe that since the generators of $Q$ are symmetric in the first
  two coordinates, so is the $Q$ itself, and consequently the
  graph $\relstr G = (A,E)$ is undirected. Clearly $[x,y]\in E$. Our
  goal is to apply Theorem~\ref{main-thm} to $\relstr G$.

  \begin{claim}
    \label{E-absorbtion}
    Consider elements
    $a_0,\ldots,a_{n-1},a'_0,\ldots,a'_{n-1}\in \{x^\alg B, y^\alg B\}$
    such that there is exactly one $i\in[:n]$ such that $a_i = a'_i$.
    Then there is a $\relstr G$-edge
    \[
    t^{\alg A}(a_0,\ldots,a_{n-1})\redgev
    t^{\alg A}(a'_0,\ldots,a'_{n-1})
    \]
  \end{claim}
  To verify the claim, we use the Taylor equation number $i$, that is
  \[
  t^{\alg B}(b_0, \ldots, b_{n-1}) = t^{\alg B}(b'_0, \ldots, b'_{n-1})
  \]
  for some $b_o,\ldots,b_{n-1},b'_0,\ldots,b'_{n-1}\in\{x^{\alg B},y^{\alg B}\}$
  where $b_i = x^{\alg B}$ and $b'_i = y^{\alg B}$.
  Since $b_i\neq b'_i$ and $a_j\neq a'_j$ for every $j\neq i$, we have
  $(a_j,a'_j,b_j,b'_j)\in Q$ for every $j\in[:n]$. Therefore
  \[
  \bigl[ t(a_0,\ldots,a_{n-1}), t(a'_0,\ldots,a'_{n-1}),
  t(b_0,\ldots,b_{n-1}), t(b'_0,\ldots,b'_{n-1}) \bigr]\in Q,
  \]
  which testify the claim.

  Due to Claim~\ref{E-absorbtion}, there is a cycle walk of length
  $2n-1$ in $\relstr G$ containing $x^{\alg A}$:
  \begin{align*}
    t^\alg A(x^\alg A,x^\alg A,\ldots,x^\alg A) &\redgev t^\alg A(x^\alg A,y^\alg A,y^\alg A,\ldots,y^\alg A)\redgev \cr
    t^\alg A(y^\alg A,x^\alg A,\ldots,x^\alg A) &\redgev t^\alg A(x^\alg A,x^\alg A,y^\alg A,\ldots,y^\alg A)\redgev \cr
    &\centerby{${}\redgev$}{$\vdots$}\cr
    t^\alg A(y^\alg A,\ldots,y^\alg A,x^\alg A) &\redgev t^\alg A(x^\alg A,x^\alg A,x^\alg A\ldots,x^\alg A).
  \end{align*}
  By Corollary~\ref{component-closed}, the component of $\relstr G$ containing $x^{\relstr A}$ is closed under $t$. However, since this component contains also $y^{\relstr A}$ and $\{x^{\relstr A},y^{\relstr A}\}$ generates $\alg A$, the component covers whole $\relstr G$, hence $\relstr G$ is connected.
  Finally, we set $\alpha_{i,i} = x^{\alg A}$ and
  $\alpha_{i,j} = y^{\alg A}$ if $i\neq j$. The edges
  \[
  x^{\alg A} \redgev
  t(y^{\alg A},\ldots,y^{\alg A}, x^{\alg A},y^{\alg A},\ldots,y^{\alg A})
  \]
  are direct consequences of Claim~\ref{E-absorbtion} and the
  idempotency of $t$.
  All the assumptions of Theorem~\ref{main-thm} are verified, so there
  is a loop $[a,a]\in E$. By definition of $E$, there is a double loop
  $[a,a,b,b]\in Q$.
\end{proof}

\begin{theorem}
  \label{double-loop-term}
  Let $\alg A$ be an idempotent algebra that satisfies Taylor
  equations locally on $X$. Then $\alg A$ satisfies double loop
  equations locally on $X$.
\end{theorem}
\begin{proof}
  We construct a ``local free algebra'' $\alg F$ with the signature of
  $\alg A$ generated by two generators. The universe $F$ of $\alg F$ consists of
  all the binary operations $X^2\to\alg A$ that can be expressed by a
  term in $\alg A$. The operations on $\alg F$ are naturally inherited from
  the basic operations on $\alg A$ by the left composition. Thus $\alg F$
  is an idempotent algebra generated by the binary projections. Let us
  denote the binary projections $x,y$ respectively.
  Since $\alg A$ satisfies some Taylor equations
  locally on $X$ and the images of the functions $x, y$ equals to
  $X$, $\alg F$ satisfies the same Taylor equations locally
  on $\{x, y\}$.

  Let $Q\subset F^4$ be a 4-ary relation on $F$ generated by all the
  quadruples $[a_0,a_1,b_0,b_1]$, where
  $a_0,a_1,b_0,b_1\in\{x, y\}$ and $a_0\neq a_1$
  or $b_0\neq b_1$.
  By Theorem~\ref{double-loop-lemma}, there is a double loop
  $[a,a,b,b]\in Q$. Therefore, there is a term $d$ in the signature of
  $\alg A$ that takes the generators of $Q$ and returns $[a,a,b,b]$.
  In particular
  \begin{align*}
  d(xx,xxxx,yyyy,yy) &= a,\\
  d(xx,yyyy,xxxx,yy) &= a,\\
  d(xy,xxyy,xxyy,xy) &= b,\\
  d(yx,xyxy,xyxy,yx) &= b.
  \end{align*}
  Thus $d$ satisfies the double loop equations if we plug in $x,y$ in
  that order. However, whenever we choose a pair $[z_0,z_1]\in X^2$,
  then $[x(z_0,z_1), y(z_0,z_1)] = [z_0, z_1]$, hence $d$ satisfies
  the double loop equations on $\alg A$ if we plug in $z_0, z_1$ in
  that order. Since $z_0, z_1$ can be any pair of elements of $X$,
  $\alg A$ satisfies the double loop equations locally on $X$.
\end{proof}

\section{Strong local loop lemma}
\label{taylor-sec}

In this section, we find a certain upgrade of the local loop lemma by
finding even weaker assumption (4) in Theorem~\ref{main-thm}.
We then use the upgraded version for
reproving a finite loop lemma for strongly connected
digraphs, in particular Theorem~7.2
in~\cite{BartoKozikLoop}

For a given $n$-ary term $t\colon A^n\to A$ and a coordinate
$i\in[:n]$, we define a digraph $\relstr P(t,i)$ on $A$ by
$x_i\redgev t(x_0,x_1,\ldots,x_{n-1})$ for all possible values
$x_0,\ldots,x_{n-1}\in A$. Using digraphs $\relstr P(t,i)$, the
assumption (4) in Theorem~\ref{main-thm} can be expressed as:
``For every $i\in[:n]$, the digraph $\relstr P(t,i)$ has a common edge
with $\relstr G$.''

Let $\overline{\relstr P(t,i)}$ denote the transitive closure of $\relstr P(t,i)$, that is en edge $u\to v$ in $\overline{\relstr P(t,i)}$ indicates a walk from $u$ to $v$ in $\relstr P(t,i)$.
Using this notation, Theorem~\ref{main-thm} has the following generalization.

\begin{theorem}
  \label{local-loop-strong}
  Consider a set $A$, operation $t\colon A^n\to A$ a digraph
  $\relstr G$ on $A$, and elements $a_i,b_i\in\relstr G$ for
  $i\in[:n]$ such that
  \begin{enumerate}
  \item $t$ is idempotent,
  \item $\relstr G$ is compatible with $t$,
  \item $\relstr G$ is either a strongly connected digraph containing
    cycle walks of all lengths greater than one, or $\relstr G$ is an undirected
    connected non-bipartite graph,
  \item for every $i \in [:n]$, there is a common edge $a_i\to b_i$ in $\relstr G$ and $\overline{\relstr P(t,i)}$.
  \end{enumerate}
  Then $\relstr G$ contains a loop.
\end{theorem}
\begin{proof}
  As in the proof of~\ref{local-loop-basic}, we denote $[:n]$ by
  $\alph$.
  By idempotency of $t$, the edges of $\relstr P(t,i)$ form a
  reflexive relation. Therefore, there is a fixed $k$ such that there
  is a $\relstr P(t,i)$-walk of length $k$ from $a_i$ to $b_i$ for
  every $i\in\alph$. For every $i\in\alph$, fix
  a substitution $f_i\colon \alph^k\to \relstr G$ 
  such that $f_i([i,i,\ldots,i]) = a_i$
  and $t^{*k}(f_i) = b_i$.

  We verify the assumptions of Theorem~\ref{main-thm} using the operation
  $t^{*(k-1)n+1}$. The operation $t^{*(k-1)n+1}$ is idempotent,
  compatible with $\relstr G$, and
  $\relstr G$ already satisfy the relational requirements.
  It remains to find the values $\alpha_{\mathbf x, \mathbf y}$,
  for every $\mathbf x,\mathbf y \in \alph^{(k-1)n+1}$ to make the
  condition (4) satisfied. We
  perceive the matrix $\alpha$ as a sequence of functions in the
  second variable, that is
  $\alpha_{\mathbf x, \mathbf y} = \alpha_\mathbf x(\mathbf y)$. We need
  to find such functions $\alpha_{\mathbf x}$ that there are a
  $\relstr G$-edges
  $$\alpha_{\mathbf x}(\mathbf x) \redgev t^{*(k-1)n+1}(\alpha_{\mathbf x}).$$

  Take $\mathbf x\in\alph^{(k-1)n+1}$. By pigeonhole principle, there
  is $i\in\alph$ occuring at least $k$-times in $\mathbf x$. Let
  $p_0,\ldots,p_{k-1}\in[:(k-1)n+1]$
  be an increasing sequence of positions in $\mathbf x$ such that
  $\mathbf x[p_j] = i$ for every $j\in[:k]$.
  We define $\alpha_{\mathbf x}$ by
  \[
  \alpha_{\mathbf x}(\mathbf y) =
  f_i([\mathbf y[p_0], \mathbf y[p_1], \ldots, \mathbf y[p_{k-1}]]).
  \]
  Thus
  \begin{align*}
    \alpha_{\mathbf x}(\mathbf x) = f_i([i,i,\ldots,i]) &= a_i,\\
    t^{*(k-1)n+1}(\alpha_\mathbf x) = t^{*k}(f_i) &= b_i,
  \end{align*}
  Therefore the assumption (4) of Theorem~\ref{main-thm} is satisfied
  by $a_i\redgev b_i$, and $\relstr G$ has a loop.
\end{proof}

From Theorem~\ref{local-loop-strong} we obtain the following finite version.

\begin{theorem}
  \label{finite-strong}
  Let $A$ be a finite set and $t\colon A^n\to A$ be an idempotent operation.
  Assume that for every $i\in[:n]$ and every pair
  $u,v\in A$, there is $w\in A$ such that there are
  edges $u\redgev w$ and $u\redgev w$ in $\overline{\relstr P(t,i)}$.
  Then every digraph $\relstr G$ that is strongly connected, compatible
  with $t$ and has algebraic length 1, has a loop.
\end{theorem}
\begin{proof}
  Fix $i\in[:n]$. We start by proving the following claim by
  induction on $|X|$
  \begin{claim}
    For every $X\subset A$, there is an element
    $b$ such that for every $x\in X$, there is an edge $x\redgev b$ in
    $\overline{\relstr P(t,i)}$.
  \end{claim}
  If $X$ is empty, it suffices to take any $b\in A$. Otherwise let
  $X = X'\cup\{x\}$, where the claim is already proven for $X'$,
  so there is $b'$ such that
  there is an edge $x'\to b$ for every $x'\in X'$. Using
  the assumption of the theorem and putting $u=b'$, $v=x$, we get a
  vertex $w = b$ such that there are edges $b'\redgev b$, $x\redgev b$.
  By transitivity of $\overline{\relstr P(t,i)}$, there are edges
  $x\redgev b$ for every $x\in X$. This finishes the proof of the
  claim.

  For every $i\in[:n]$, we fix $b_i\in A$ such that there is an edge $x\to b_i$ in $\overline{\relstr P(t,i)}$ for every $x\in A$.
  Consider a strongly connected digraph $\relstr G$ with
  algebraic length 1 that is compatible with $t$. To obtain a
  contradiction, suppose that there is no cycle walk of length 1 (a loop) in
  $\relstr G$. Since $\relstr G$ is a strongly connected digraph with
  algebraic length 1, it contains cycle walks of all enough large
  lengths. Let $k$ denote the largest length such that there is no
  cycle walk of length $k$ in $\relstr G$. The relational power
  $\relstr G^{\relpow k}$ is compatible
  with $t$, strongly connected, and by the choice of $k$,
  $\relstr G^{\relpow k}$ contains cycle walks of all lengths
  greater than 1 but no loop.

  Since $\relstr G^{\relpow k}$ is strongly connected, we can find nodes $a_i$ such
  that there are edges $a_i\to b_i$ in $\relstr G^{\relpow k}$. Any such edge is also an edge in $\overline{\relstr P(t,i)}$ by the choice of $b_i$.
  Therefore, the assumptions of
  Theorem~\ref{local-loop-strong} are satisfied, and we get a
  contradiction with the assumption that $\relstr G^{\relpow k}$ has no loop.
\end{proof}

The standard finite loop lemma for strongly connected digraphs,
originally proved in~\cite{BartoKozikLoop}, is a direct consequence.

\begin{corollary}
  Let $\relstr G$ be a strongly connected digraph with algebraic
  length 1 compatible with a Taylor operation. Then $\relstr G$ has
  a loop.
\end{corollary}
\begin{proof}
  Let us denote the Taylor operation as $t$, and the vertex set of
  $\relstr G$ as $A$.
  Since the digraph $\relstr G$ is strongly connected and has algebraic
  length 1, it remains to verify the assumptions of
  Theorem~\ref{finite-strong}. We take $i\in[:n]$
  and $u,v\in A$, and find $w$ such that
  $u\redgev w$ and $v\redgev w$ in $\relstr P(t,i)$.
  This is straightforward, it suffices to set
  \[
  w = t(x_0,\ldots,x_i=u,\ldots,x_{n-1}) = t(y_0,\ldots,y_i=v,\ldots,y_{n-1}),
  \]
  where $x_j,y_j$ are set to $u$ or $v$ according to the Taylor
  equation number $i$.
\end{proof}

\section{Conclusion}
\label{conclusion}

Though perceiving positions of variables in a star power as words and
applying a simple word combinatorics on them gives surprisingly strong
results, further research is needed.
In particular, we would like to see a proof that is able to
separate the technical effort from the overall powerful
machinery. This could lead not only to a nicer proof the local loop lemma in
this article but also pave a way to various interesting generalizations.

A modest generalization would be replacing the item (3) in
Theorem~\ref{main-thm} by simply $\relstr G$ being a strongly
connected digraph with algebraic length 1. The fact that we were able
to get around it whenever we needed suggests that it is really rather
a technical issue in the proof rather than a real obstacle.

A bit bolder attempt would be replacing the assumption of a strongly
connected graph by something weaker. While the finite loop lemma,
Theorem~\ref{dir-loop-lemma}, suggests that the assumption of a strongly
connected digraph is not entirely necessary, the strong connectedness
forms a solid barrier for loop conditions, see Chapter 6
in~\cite{LoopConditions} for counterexamples.

However, to get a widely applicable and powerful tool, it is necessary
to get beyond a single digraph. What are the necessary assumptions to
get a loop shared by two digraphs? What about loops in hypergraphs
(see~\cite{PseudoLoopConditions})?
A natural question comes from Section~\ref{double-loop}. While
it is possible to prove that a local Taylor implies local double-loop
term, there is a much simpler form of the (global) weakest non-trivial
idempotent equational condition:
\begin{question}
Let $\alg A$ be an idempotent algebra that satisfies Taylor identities
locally on a set $X$. Does it necessarily has a term
that satisfies
$$
t(x,y,y,y,x,x) = t(y,x,y,x,y,x) = t(y,y,x,x,x,y)
$$
locally on $X$?
\end{question}

And last but not least, is it possible to apply the ideas in this
article to oligomorphic structures, and consequently, infinite
constraint satisfaction problem
(see~\cite{PseudoSiggers,PseudoLoopConditions})?
Oligomorphic algebras are kind of the opposite of idempotent
algebras -- idempotent algebras have only the trivial unary term
operation while oligomorphic algebras has a large group of them. On
the other hand, notice that the proof of Theorem~\ref{main-thm}
uses the idempotency at just two position in a predictable manner,
therefore there might be a way of using a variant of the local loop
lemma in algebras that are not idempotent.

\bibliographystyle{plain}
\bibliography{bib-file.bib}

\end{document}